\documentclass[12pt]{amsart}
\usepackage{verbatim}
\usepackage{amssymb}
\usepackage{upref}
\usepackage[all]{xy}
\usepackage[usenames,dvipsnames]{color}
\usepackage[normalem]{ulem}
\usepackage{amsthm}

\allowdisplaybreaks

\newtheorem{thm}{Theorem}[section]
\newtheorem{lem}[thm]{Lemma}
\newtheorem{cor}[thm]{Corollary}
\newtheorem{prop}[thm]{Proposition}

\theoremstyle{definition}
\newtheorem{defn}[thm]{Definition}
\newtheorem{q}[thm]{Question}
\newtheorem{ex}[thm]{Example}
\newtheorem{hyp}[thm]{Hypothesis}

\newtheorem{rem}[thm]{Remark}

\numberwithin{equation}{section}

\newcommand{\secref}[1]{Section~\textup{\ref{#1}}}

\newcommand{\corref}[1]{Corollary~\textup{\ref{#1}}}
\newcommand{\lemref}[1]{Lemma~\textup{\ref{#1}}}
\newcommand{\propref}[1]{Proposition~\textup{\ref{#1}}}
\newcommand{\remref}[1]{Remark~\textup{\ref{#1}}}
\newcommand{\exref}[1]{Example~\textup{\ref{#1}}}

\newcommand{\C}{\mathbb C}
\newcommand{\T}{\mathbb T}
\newcommand{\Z}{\mathbb Z}

\newcommand{\OO}{\mathcal O}
\newcommand{\TT}{\mathcal T}
\newcommand{\HH}{\mathcal H}

\newcommand{\KK}{\mathcal K}
\newcommand{\LL}{\mathcal L}

\newcommand{\scr}{\mathcal}

\newcommand{\inv}{^{-1}}
\newcommand{\midtext}[1]{\quad\text{#1}\quad}
\newcommand{\righttext}[1]{\quad\text{#1 }}
\newcommand{\under}{\backslash}
\renewcommand{\:}{\colon}
\newcommand{\case}{& \text{if }}
\newcommand{\ifnot}{& \text{otherwise}}
\newcommand{\<}{\langle}
\renewcommand{\>}{\rangle}
\newcommand{\what}{\widehat}
\newcommand{\wilde}{\widetilde}
\renewcommand{\epsilon}{\varepsilon}

\renewcommand{\bar}{\overline}

\newcommand{\mtx}[1]{\begin{pmatrix} #1 \end{pmatrix}}

\newcommand{\smtx}[1]{\left(\begin{smallmatrix} #1\end{smallmatrix}\right)}

\newcommand{\ann}{^\perp}

\DeclareMathOperator{\ad}{Ad}
\DeclareMathOperator{\ind}{Ind}
\DeclareMathOperator*{\spn}{span}
\DeclareMathOperator*{\clspn}{\overline{\spn}}
\DeclareMathOperator{\dashind}{\!-Ind}
\DeclareMathOperator{\aut}{Aut}
\DeclareMathOperator{\ran}{ran}
\DeclareMathOperator{\supp}{supp}
\DeclareMathOperator{\coker}{coker}

\begin{document}
\title{Subgroup correspondences}
\author{S. Kaliszewski}
\address{School of Mathematical and Statistical Sciences, Arizona State University Tempe, Arizona, USA 85287}
\email{kaliszewski@asu.edu}
\author{Nadia S. Larsen}
\address{Department of Mathematics, University of Oslo,
PO BOX 1053 Blindern, N-0316 Oslo, Norway}
\email{nadiasl@math.uio.no}
\author{John Quigg}
\address{School of Mathematical and Statistical Sciences, Arizona State University Tempe, Arizona, USA 85287}
\email{quigg@asu.edu}
\date{December 14, 2016. Revised October 17, 2017.}


\dedicatory{Dedicated to the memory of Ola Bratteli}

\begin{abstract}
For a closed subgroup of a locally compact group the Rieffel induction process gives rise to a $C^*$-correspondence
over the $C^*$-algebra of the subgroup.
We study the associated Cuntz-Pimsner algebra  and show that, by varying the subgroup to be
 open, compact, or discrete,  there are connections with the Exel-Pardo correspondence arising from a cocycle,
and also with graph algebras.
\end{abstract}
\maketitle

\section{Introduction}\label{intro}

Let $H$ be a closed subgroup of a locally compact group $G$.
Then the Rieffel induction process involves a $C^*(G)-C^*(H)$ correspondence $X$,
and restricting to $H$ (more precisely, composing the left $C^*(G)$-module structure on $X$ with the canonical homomorphism from $C^*(H)$ into $M(C^*(G))$) makes $X$ into a correspondence over $C^*(H)$.
We examine properties of the Cuntz-Pimsner algebra of this correspondence in terms of how $H$ sits as a subgroup of $G$.

The $C^*(H)$-correspondence $X$ has some special properties, e.g., it is nondegenerate and full.
Our results are sharpest when $X$ is \emph{regular}, i.e., $C^*(H)$ acts on the left faithfully by compact operators,
which seems to entail $H$ being open and of finite index in $G$.
In this case the representations of the Cuntz-Pimsner algebra $\OO_X$ correspond to representations $U$ of $H$ together with an explicit unitary equivalence between $U$ and $(\ind_H^G U)|_H$.
If $H$ is open and central in $G$, then the Cuntz-Pimsner algebra $\OO_X$ is the tensor product of $C^*(H)$ and a Cuntz algebra.

When $G$ is discrete, any choice of cross section of $G/H$ in $G$ gives rise to a cocycle for the action of $H$ on $G/H$ by translation, and $\OO_X$ is isomorphic to an associated Exel-Pardo algebra (for the action of $H$ on a directed graph with one vertex), generated by a Cuntz algebra and a representation of $H$ whose interaction with the Cuntz algebra involves the cocycle. Alternatively, this is described by a self-similar action of $H$ in the sense of Nekrashevych.
The cohomology class of the cocycle seems to be determined by the subgroup $H$ itself, explaining the independence of $\OO_X$ upon the choice of cross section.

When the subgroup $H$ is compact the Peter-Weyl theorem says that $C^*(H)$ is a $c_0$-direct sum of finite-dimensional algebras,
so is Morita equivalent to a commutative $C^*$-algebra with the same spectrum as $H$.
It follows that, by a theorem of Muhly and Solel, $\OO_X$ is Morita equivalent to the Cuntz-Pimsner algebra of a correspondence over this commutative algebra, and hence (by a result of Patani and the first and third authors) to the $C^*$-algebra of a directed graph $E$ that can be computed in terms of multiplicities of irreducible representations of $H$ induced across the correspondence $X$.
If $H$ is already abelian, then $\OO_X$ is isomorphic to this graph algebra $C^*(E)$.

In \secref{ex} we specialize further to a finite group $G$.
Then Mackey's Subgroup Theorem allows us to compute the multiplicities (and hence the directed graph $E$)
using the double $H$-cosets.
It turns out that interesting examples arise even when $H$ has order 2, and we examine this case in some detail.
$C^*(E)$, and hence $\OO_X$, is a UCT Kirchberg algebra that is classifiable by its
$K$-theory,
which only depends upon how large the centralizer of $H$ is in $G$,
more precisely upon two positive integers $r$ and $q$,
where the first is the index of $H$ in its centralizer and $r+2q$ is the index of $H$ in $G$.
When $r=1$ we have
$K_0=\Z_q\oplus \Z$ and $K_1=\Z$,
and it follows
(taking into account also the class of the identity in $K_0$)
that $\OO_X$ is isomorphic to
the $C^*$-algebra of the category of paths
given by the positive submonoid of a Baumslag-Solitar group,
studied by Spielberg.
When $q$ is also 1, $\OO_X$ is Morita equivalent to
two $C^*$-algebras studied by Laca and Spielberg,
involving a projective linear group acting on the boundary of the upper half plane or alternatively the Ruelle algebra of a 2-adic solenoid.
On the other hand, when $r>1$
the $K_1$ group is trivial,
and the $K_0$ group depends upon whether $r+q-1$ and $q$ are coprime.
If they are coprime, the $K_0$ group is finite cyclic, and hence
$\OO_X$ is a matrix algebra over a Cuntz algebra.
But if $r+q-1$ and $q$ are not coprime then $K_0$ is a direct sum of two finite cyclic groups,
and unfortunately we do not know any other famous Kirchberg algebras with this $K$-theory.

In the last section we briefly discuss a curious connection with Doplicher-Roberts algebras studied by Mann, Raeburn, and Sutherland.
The situation is decidedly different (in particular, not involving induced representations), but the results are uncannily similar.

We thank Jack Spielberg for numerous helpful conversations. This research was initiated during the visit of the second named author to Arizona State University and she thanks her two collaborators and Jack Spielberg for their hospitality.
Some of this research was done during a visit of the third author to the University of Oslo, and he thanks Erik B\'edos, Nadia Larsen, and Tron Omland for their hospitality. We are grateful to the anonymous referee for many useful suggestions that significantly improved the paper.

\section{Preliminaries}\label{prelim}

We record our notation and conventions for $C^*$-correspondences.
First of all, if $X$ is an $A$-correspondence, with left $A$-module structure given by the homomorphism $\varphi=\varphi_A\:A\to \LL(X)$, we will freely switch back and forth between the notations $ax$ and $\varphi(a)x$ for $a\in A,x\in X$.
We call the the correspondence $X$ \emph{faithful} if $\varphi$ is faithful, and \emph{nondegenerate} if $\varphi(A)X=X$.

A (Toeplitz) representation of $X$ in a $C^*$-algebra $B$ is a pair $(\psi,\pi)$, where
$\psi\:X\to B$ is a linear map and $\pi\:A\to B$ is a homomorphism such that
for $a\in A,x,y\in X$ we have
\begin{align*}
\psi(ax)&=\pi(a)\psi(x)
\\
\psi(x)^*\psi(y)&=\pi\bigl(\<x,y\>_{A}\bigr)
\end{align*}
(and hence $\psi(xa)=\psi(x)\pi(a)$).
If $\HH$ is a Hilbert space and $B$ is the algebra $B(\HH)$ of bounded operators on $\HH$, we say $(\psi,\pi)$ is a representation of $X$ on $\HH$.
A representation $(\psi,\pi)$ of $X$ on a Hilbert space $\HH$ is \emph{nondegenerate} if the $C^*$-algebra generated by $\psi(X)\cup \pi(A)$ acts nondegenerately on $\HH$.
If $X$ is nondegenerate as a correspondence then a representation $(\psi,\pi)$ of $X$ on $\HH$ is nondegenerate if and only if the representation $\pi$ of $A$ is nondegenerate.

The \emph{Toeplitz algebra} $\TT_X$ of $X$ is universal for Toeplitz representations.
$\KK(X)$ denotes the algebra of (generalized) compact operators on $X$, which is the closed linear span of the (generalized) rank-one operators $\theta_{x,y}$ given by $\theta_{x,y}z=x\<y,z\>_A$.
For any representation $(\psi,\pi)$ of $X$ in $B$, there is a unique homomorphism $\psi^{(1)}\:\KK(X)\to B$ such that $\psi^{(1)}(\theta_{x,y})=\psi(x)\psi(y)^*$ for all $x,y\in X$.

The \emph{Katsura ideal} of $A$ is $J_X:=\varphi\inv(\KK(X))\cap (\ker\varphi)\ann$,
where for any ideal $I$ of $A$ the \emph{orthogonal complement} of $I$ is the ideal $I\ann:=\{a\in A:ab=0\text{ for all }b\in I\}$.
A representation $(\psi,\pi)$ of $X$ in $B$ is \emph{Cuntz-Pimsner covariant} if $\pi(a)=\psi^{(1)}\circ\varphi(a)$ for all $a\in J_X$,
and the \emph{Cuntz-Pimsner algebra} $\OO_X$ is universal for Cuntz-Pimsner covariant representations, and is generated as a $C^*$-algebra by a universal Cuntz-Pimsner covariant representation $(k_X,k_A)$.
For any Cuntz-Pimsner covariant representation $(\psi,\pi)$ of $X$, we write $\psi\times\pi$ for the unique homomorphism of $\OO_X$ satisfying
\[
\psi=(\psi\times\pi)\circ k_X\midtext{and}\pi=(\psi\times\pi)\circ k_A.
\]
If $X$ is nondegenerate as an $A$-correspondence,
then the homomorphism $k_A\:A\to \OO_X$ is nondegenerate in the sense that $k_A(A)\OO_X=\OO_X$.

Our primary object of study will be the Cuntz-Pimsner algebra of a correspondence over the $C^*$-algebra of a subgroup $H$ of a locally compact group $G$.
Thus it is relevant to consider what sorts of representations of $H$ will occur as part of a Cuntz-Pimsner covariant representation.
The remainder of this section is devoted to some general remarks concerning representations of $C^*$-correspondences.
We claim no originality for these --- they are either readily available in the literature, or folklore. We refer to \cite[\S 2.4]{tfb} for background on induced representations.

\begin{lem}\label{ind res}
The
Toeplitz representations of
an
$A$-correspondence $X$
on a Hilbert space $\HH$
are in 1-1 correspondence with the
pairs $(\pi,V)$, where $\pi$ is a representation of $A$ on $\HH$
and $V\:X\otimes_A \HH\to \HH$ is an isometry implementing a unitary equivalence between
$X\dashind\pi$ and a subrepresentation of $\pi$.
\end{lem}

\begin{proof}
Let $(\psi,\pi)$ be a representation of $X$ on $\HH$.
The Rieffel induction process yields a representation
of
$\LL(X)$ on $B(X\otimes_A\HH)$,
and composing with the left-module homomorphism $\varphi\:A\to \LL(X)$ gives an induced representation $X\dashind\pi\:A\to B(X\otimes_A \HH)$.

Borrowing from \cite{FowRae},
we can define an isometry $V\:X\otimes_A \HH\to \HH$ by
\[
V(x\otimes_A \xi)=\psi(x)\xi\righttext{for}x\in X,\xi\in \HH.
\]
Conjugating by $V$,
\cite[Proposition~1.6]{FowRae} gives a unique representation
$\rho\:\LL(X)\to B(\HH)$
with essential subspace
\[
\HH_\psi:=\clspn\{\psi(X)\HH\}=\ran V
\]
such that
\[
\rho(T)\psi(x)\xi=\psi(Tx)\xi\righttext{for}T\in \LL(X),x\in X,\xi\in \HH,
\]
and moreover $\rho(\theta_{x,y})=\psi(x)\psi(y)^*$.

A quick computation shows that
the diagram
\[
\xymatrix@C+30pt{
X\otimes_A\HH \ar[r]^-{X\dashind\pi(a)} \ar[d]_V
&X\otimes_A\HH \ar[d]^V
\\
\HH \ar[r]_-{\pi(a)}
&\HH
}
\]
commutes. This also shows that $(\psi,\pi)$ is Cuntz-Pimsner covariant on the invariant subspace $\HH_\psi$,
because if $\varphi(a)\in \KK(X)$ then for all $x\in X$ we have
\[
\psi^{(1)}(\varphi(a))\psi(x)=\rho(\varphi(a))\psi(x)=\psi(ax)=\pi(a)\psi(x).
\]

Thus
$V$ implements a unitary equivalence between $X\dashind\pi$ and a subrepresentation of $\pi$,
namely the restriction of $\pi$ to $\HH_\psi$.

For the converse,
suppose we have a representation $\pi\:A\to B(\HH)$
and an isometry $V\:X\otimes_A \HH\to \HH$,
with range $L$,
such that
\[
\ad V\circ X\dashind\pi(a)\xi=\pi(a)\xi\righttext{for all}a\in A,\xi\in L.
\]
We must show that there exists a linear map $\psi:X\to B(\HH)$ such that
 $(\psi,\pi)$ is a Toeplitz representation and $V(x\otimes_A\xi)=\psi(x)\xi$ for $x\in X,\xi\in\HH$.
For $x\in X$, define $\psi(x)\:\HH\to\HH$ by
\[
\psi(x)\xi=V(x\otimes_A\xi).
\]
Using that $V$ is isometric, it follows that $(\psi,\pi)$ is a Toeplitz representation of $X$. We omit the details.

Finally, it is obvious from the constructions that if we now
start with this newly manufactured $\psi$
then the intertwining isometry defined as in the first part of the proof agrees with $V$.
\end{proof}

Before considering the Cuntz-Pimsner covariant representations,
we specialize the correspondence:

\begin{defn}
We call an $A$-correspondence $X$ \emph{regular} if $J_X=A$,
i.e., $A$ acts faithfully by compact operators on $X$.
\end{defn}

\begin{rem}
If $X$ is nondegenerate and regular, then $k_X^{(1)}\:\KK(X)\to \OO_X$ is nondegenerate,
because $k_A$ is.
\end{rem}

Recall that for a representation $(\psi,\pi)$ of $X$ on a Hilbert space $\HH$ we write $\HH_\psi=\clspn\{\psi(X)\HH\}$.

After we had completed this paper, we learned that the following lemma is essentially the same as \cite[Proposition 2.5]{AlbMey}.

\begin{lem}
A nondegenerate representation $(\psi,\pi)$ of a nondegenerate regular $A$-correspondence $X$
on a Hilbert space $\HH$
is
Cuntz-Pimsner covariant if and only if $\HH_\psi=\HH$.
\end{lem}

\begin{proof}
First assume that
$\HH_\psi=\HH$.
By \cite[Proposition~1.6 (1)]{FowRae} there is a
unique representation $\rho\:\LL(X)\to B(\HH)$,
with essential subspace $\HH_\psi$,
such that
\[
\rho(S)\psi(x)\xi=\psi(Sx)\xi\righttext{for}S\in \LL(X),x\in X,\xi\in \HH,
\]
and moreover the restriction of $\rho$ to $\KK(X)$ coincides with the canonical representation $\psi^{(1)}$.
For $a\in J_X$,
we must show that $\pi(a)=\psi^{(1)}\circ \varphi(a)$.
Since $\HH_\psi=\HH$,
by density and continuity it suffices to note that
for all $x\in X$ and $\xi\in \HH$,
since $\varphi(a)\in \KK(X)$,
\begin{align*}
\psi^{(1)}\circ \varphi(a)\psi(x)\xi
&=\rho(\varphi(a))\psi(x)\xi
\\&=\psi(\varphi(a)x)\xi
\\&=\psi(ax)\xi
\\&=\pi(a)\psi(x)\xi.
\end{align*}

Conversely, assume that $(\psi,\pi)$ is Cuntz-Pimsner covariant.
Since
$\pi$ is nondegenerate,
it suffices to show that
\[
\pi(A)\subset \clspn\{\psi(X)\psi(X)^*\}.
\]
But since $J_X=A$, by Cuntz-Pimsner covariance we have
\begin{align*}
\pi(A)\subset \psi^{(1)}(\KK(X))
&=\clspn\{\psi^{(1)}(\theta_{x,y}):x,y\in X\}
\\&=\clspn\{\psi(X)\psi(X)^*\}.
\qedhere
\end{align*}
\end{proof}

\begin{cor}\label{ind res reg}
The Cuntz-Pimsner covariant representations of a nondegenerate regular $A$-correspondence $X$ on a Hilbert space $\HH$ are in 1-1 correspondence with the pairs $(\pi,V)$, where $\pi$ is a representation of $A$ on $\HH$
and $V\:X\otimes_A\HH\to\HH$ implements a unitary equivalence between $X\dashind\pi$ and $\pi$.
\end{cor}

\section{Subgroups}\label{subgroups}

Now let $H$ be a closed subgroup of a locally compact group $G$,
and let $X$ be the
$C^*(G)-C^*(H)$ correspondence for Rieffel induction, see for example \cite[\S 2.4 and Appendix C]{tfb}. We will assume henceforth that $G$ is second countable.
Composing the left $C^*(G)$-module structure with the canonical nondegenerate homomorphism
$C^*(H)\to M(C^*(G))$,
$X$ becomes a
$C^*$-correspondence over $A:=C^*(H)$

Note that the left-module homomorphism $\varphi=\varphi_A\:A\to \LL(X)$ is nondegenerate, and the $A$-correspondence $X$ is full.
It still seems to be unknown (at least to us) whether
the correspondence $X$ is always \emph{faithful} in the sense that
$\varphi_A$ is faithful,
equivalently whether the canonical homomorphism $C^*(H)\to M(C^*(G))$ is faithful
(see \cite[paragraph following Proposition~4.1]{rie:induced}.
It is faithful if the subgroup $H$ is either open \cite[Proposition~1.2]{rie:induced} or compact (this follows from \cite[Corollary~3 of Theorem~5.5]{fel:weak2}),
and also if $H$ is amenable,
since then $C^*(H)=C^*_r(H)$ and the composition
\[
C^*_r(H)\to M(C^*(G))\to M(C^*_r(G))
\]
is always faithful.
It seems to us that examples where $C^*(H)\to M(C^*(G))$ is not faithful, if they exist,
will be somewhat exotic.

\begin{hyp}
We will tacitly assume throughout that
the subgroup $H$ of $G$ is such that
the correspondence $X$ is faithful,
equivalently $C^*(H)\to M(C^*(G))$ is faithful.
\end{hyp}

\begin{q}
When will $\varphi_A$ map $C^*(H)$ into the algebra $\KK(X)$ of compact operators?
\end{q}

Note that the imprimitivity theorem, cf. e.g. \cite[Theorem C.23]{tfb} says
\[
\KK(X)=C_0(G/H)\rtimes G.
\]
If $H$ is open then
the natural inclusion $C_c(H)\hookrightarrow C_c(G)$
extends to a faithful embedding
$C^*(H)\subset C^*(G)$ \cite[Proposition~1.2]{rie:induced}.
If $H$ is cocompact in $G$, i.e., $G/H$ is compact, then $C_0(G/H)=C(G/H)$ is unital, so
$i_G(C^*(G))\subset C(G/H)\rtimes G$.
So, if $H$ is open and cocompact then $\varphi(A)\subset \KK(X)$.

On the other hand, if $H$ is not cocompact, then $C_0(G/H)$ is not unital,
and it follows from \lemref{nonunital} below
below that $\varphi(A)\cap \KK(X)=\{0\}$.
If $H$ is cocompact but not open,
the situation is not clear to us in general,
and we will not seriously study this case.

In the preceding paragraph we appealed to the following lemma, which must be folklore:

\begin{lem}\label{nonunital}
If $\alpha$ is an action of a locally compact group $G$ on a nonunital $C^*$-algebra $A$, then
\[
i_G(M(C^*(G)))\cap (A\rtimes_\alpha G)=\{0\}.
\]
\end{lem}

\begin{proof}
First note that it suffices to show that $i_G(C^*(G))\cap (A\rtimes G)=\{0\}$,
because then if we had any nonzero $m\in M(C^*(G))$ for which $i_G(m)\in A\rtimes G$,
then we could choose $c\in C^*(G)$ such that $mc\ne 0$,
and then $i_G(mc)$ would be a nonzero element of $i_G(C^*(G))\cap (A\rtimes G)$.

The action extends continuously to the unitization $\wilde A$,
and we have a split short exact sequence
\[
\xymatrix{
0\ar[r]
&A \ar[r]^\iota
&\wilde A \ar@/_/[r]_\rho
&\C \ar@/_/[l]_\sigma \ar[r]
&0
}
\]
that is $G$-equivariant.
Taking crossed products, we get a split short exact sequence
\[
\xymatrix{
0\ar[r]
&A\rtimes G \ar[r]^{\iota\rtimes G}
&\wilde A\rtimes G \ar@/_/[r]_{\rho\rtimes G}
&C^*(G) \ar@/_/_-{i_G^{\wilde A}}[l] \ar[r]
&0,
}
\]
where $i_G^{\wilde A}\:C^*(G)\to \wilde A\rtimes G$ is the canonical homomorphism,
which coincides with
$\sigma\rtimes G$.

The canonical covariant pair $(i_A,i_G^A)\:(A,G)\to M(A\rtimes_\alpha G)$
is compatible with the pair $(i_{\wilde A},i_G^{\wilde A})$ in the following sense:
first of all, the nondegenerate homomorphism
$i_A\:A\to M(A\rtimes G)$
extends canonically to
$\bar i_A\:\wilde A\to M(A\rtimes G)$,
the pair
$(\bar i_A,i_G^A)\:(\wilde A,G)\to M(A\rtimes G)$
is covariant
and
the diagram
\[
\xymatrix@C+20pt{
A\rtimes G \ar[r]^-{\iota\rtimes G} \ar@{^(->}[d]
&\wilde A\rtimes G \ar[dl]_{\bar i_A\times i_G^A}
\\
M(A\rtimes G)
&C^*(G) \ar[l]^-{i_G^A} \ar[u]_{i_G^{\wilde A}}
}
\]
commutes.
Combining diagrams, if we had a nonzero $c\in C^*(G)$ such that $i_G^A(c)\in A\rtimes G$,
then $i_G^{\wilde A}(c)$ would be a nonzero element of $\wilde A\rtimes G$
that lies in the ideal $A\rtimes G$,
which would give
\[
0\ne c=(\rho\rtimes G)\circ i_G^{\wilde A}(c)=(\rho\rtimes G)\circ (\iota\rtimes G)\circ i_G^A(c)=0,
\]
a contradiction.
\end{proof}

\begin{cor}
When $H$ is open,
we have a dichotomy:
$G/H$ is
either finite, in which case $J_X=A$,
or infinite, in which case $J_X=\{0\}$
and $\OO_X=\TT_X$.
\end{cor}

\begin{rem}
In any case, if $H$ is cocompact in $G$
and $(\psi,\pi)$ is a Toeplitz representation of the $A$-correspondence $X$ on a Hilbert space $\HH$,
then
for $a\in A,x\in X$ we have
\[
\pi(a)\psi(x)=\psi(ax)=\psi^{(1)}(\varphi(a))\psi(x),
\]
so the restriction of $(\psi,\pi)$ to the invariant subspace $\HH_\psi$ is Cuntz-Pimsner covariant.
\end{rem}

Here are the two extremes for how $H$ can sit inside $G$:
if $H=\{1\}$, then $X$ is the Hilbert space $L^2(G)$,
regarded as a $\C$-correspondence,
so $\OO_X$ is the Cuntz algebra $\OO_{L^2(G)}$.
Note that, due to our standing hypothesis that $G$ is second countable,
the Hilbert space $L^2(G)$ is separable,
and so $\OO_{L^2(G)}$ is either the Cuntz algebra $\OO_n$ if $G$ is finite of order $n$, or $\OO_\infty$ if $G$ is infinite.
At the other extreme,
if $H=G$, then $X$ is the \emph{identity} $C^*(G)$-correspondence $C^*(G)$, so $\OO_X=C(\T)\otimes C^*(G)$.

Here are a couple of obvious general properties of $X$ and $\OO_X$.
If $H$ is exact, then
so is $C^*(H)$, so
$\OO_X$ is exact
by \cite[Theorem~7.1]{katsuracorrespondence}.
Similarly,
if $H$ is amenable, or more generally if $C^*(H)$ is nuclear, then $\OO_X$ is nuclear,
by
\cite[Corollary~7.4]{katsuracorrespondence}.

\section{$H$ open}\label{open}

Suppose that $H$ is an open subgroup of $G$.
Then every double $H$-coset $HtH$ is open,
and $C_c(HtH)$ is closed under left and right multiplication by $C_c(H)$
(in the convolution algebra $C_c(G)$).
Note that $C_c(G)$ is the algebraic direct sum of the vector subspaces $C_c(HtH)$.
If $f\in C_c(HtH)$ and $g\in C_c(HsH)$
then
\begin{align*}
\<f,g\>_A(h)
&=(f^**g)(h)
\\&=\int_G f^*(r)g(r\inv h)\,dr
\\&=\int_G \bar{f(r\inv)}\Delta(r\inv)g(r\inv h)\,dr
\\&=\int_G \bar{f(r)}g(rh)\,dr
\\&=\int_{HtH} \bar{f(r)}g(rh)\,dr
\end{align*}
which is 0 unless
$HtH=HsH$.
It follows that the norm closures $X_{HtH}$ of the sets $C_c(HtH)$ in $X$
are mutually orthogonal $A$-subcorrespondences,
and we get a decomposition
\[
X=\bigoplus_{HtH\in H\under G/H}X_{HtH}
\]
of correspondences.
This might be of use in later investigations, but at present we only apply it to the following special case.

\begin{prop}\label{normal}
Let $H$ be open and normal in $G$.
Choose a cross section $\eta\:G/H\to G$, with $\eta(H)=1$.
Let $A=C^*(H)$.
For each $tH\in G/H$,
let $A_{tH}$ be the $A$-correspondence associated to the automorphism $\ad \eta(tH)\inv$
of $A$,
i.e., it is the standard Hilbert $A$-module $A$ but with left $A$-module structure given by
\[
a\cdot_{tH} b=\ad \eta(tH)\inv(a)b\righttext{for}a,b\in A.
\]
Then
\[
X\simeq \bigoplus_{tH\in G/H}A_{tH}.
\]
as $A$-correspondences.
\end{prop}

\begin{proof}
More precisely, for $s\in G$ the associated automorphism $\ad s$ of $A$ involves the modular function: if $f\in C_c(H)$ then $\ad s(f)$ is the function in $C_c(H)$ defined by
\[
(\ad s(f))(h)=f(s\inv hs)\Delta(s).
\]
Nore that since $H$ is open and normal in $G$, the modular function and Haar measure on $H$ are the restrictions of those on $G$.
Since $H$ is normal, the double cosets $HtH$ are just cosets $tH$,
so by the discussion preceding the proposition
we have a decomposition
\[
X=\bigoplus_{tH\in G/H}X_{tH},
\]
where $X_{tH}$ is the closure of $C_c(tH)$ in $X$.
It now suffices to show that
for all $tH\in G/H$ we have
$X_{tH}\simeq A_{tH}$ as $A$-correspondences.

We use the conventions from \cite[Theorem~C.23]{tfb} for the correspondence $X$;
more precisely, the formulas in \cite{tfb} are for the $(C_0(G/H)\rtimes G)-C^*(H)$ imprimitivity bimodule structure on $X$,
and we can restrict the left module multiplication to $C^*(G)$, and then we restrict further to the subalgebra $A=C^*(H)$.
Define a linear bijection $\Psi\:C_c(tH)\to C_c(H)$ by
\[
(\Psi x)(h)=x\bigl(\eta(tH)h\bigr).
\]
Then for $k\in C_c(H)$, $x\in C_c(tH)$, and $h\in H$ we have,
letting $s=\eta(tH)$,
\begin{align*}
(\Psi kx)(h)
&=(kx)(sh)
\\&=\int_G k(r)x(r\inv sh)\,dr\righttext{by (C.21) of \cite{tfb}}
\\&=\int_H k(r)x(r\inv sh)\,dr\righttext{since $\supp k\subset H$}
\\&=\int_H k(srs\inv)\Delta(s\inv)x(sr\inv h)\,dr
\\&=\int_H \ad s\inv(k)(r)x(sr\inv h)\,dr
\\&=\int_H \ad s\inv(k)(r)(\Psi x)(r\inv h)\,dr
\\&=\bigl(\ad s\inv(k)\Psi x\bigr)(h)\righttext{convolution product in $C_c(H)$}
\\&=\bigl(k\cdot_{tH} \Psi x\bigr)(h).
\end{align*}
Thus $\Psi$ preserves the left $C_c(H)$-module structure.

For the inner products,
if $x,y\in C_c(tH)$ and $h\in H$ then,
by \cite[Equation~(C.20)]{tfb},
since the modular quotient function $\gamma$ in that reference is identically 1 since the modular function on $H$ is the restriction of the one on $G$,
we have
\begin{align*}
\<\Psi x,\Psi y\>_A(h)
&=\int_G \bar{(\Psi x)(r)}(\Psi y)(rh)\,dr
\\&=\int_G \bar{x(\eta(tH)r)}y(\eta(tH)rh)\,dr
\\&=\int_G \bar{x(r)}y(rh)\,dr
\\&=\<x,y\>_A(h).
\end{align*}
Since $\Psi$ preserves the inner products and the left module structure, it automatically preserves the right $C_c(H)$-module structure.
Thus $\Psi$ extends by continuity to an isomorphism of $A$-correspondences, and we are done.
\end{proof}

And now we specialize even further:

\begin{cor}\label{central}
If $H$ is open and central in $G$, then
\[
X\simeq \ell^2(G/H)\otimes A,
\]
where on the right-hand side we mean the external tensor product of the $\C$-correspondence $\ell^2(G/H)$ and the standard $A$-correspondence $A$.
Consequently, by \lemref{tensor comm} below,
\[
\OO_X\simeq \OO_{\ell^2(G/H)}\otimes A.
\]
\end{cor}

\corref{central} referred to
the following lemma, which is probably folklore, although we could not find a convenient reference for it.
It could almost (but not quite) be deduced from \cite[Theorem~5.4]{adamtensor}, but our special case is much more elementary.

\begin{lem}\label{tensor comm}
Let $A$ be a
$C^*$-algebra, let $\HH$ be a Hilbert space,
and let $\HH\otimes A$ be the $A$-correspondence given by the external tensor product of the $\C$-correspondence $\HH$ and the standard $A$-correspondence $A$.
Then
\[
\OO_{\HH\otimes A}\simeq \OO_\HH\otimes A.
\]
\end{lem}

\begin{proof}
Tensoring the universal Cuntz-Pimsner covariant representation of $\HH$ in $\OO_\HH$ with the identity map on $A$ gives a Toeplitz representation of $\HH\otimes A$ in $\OO_\HH\otimes $A. If $\dim(\HH)=\infty$, then $J_{\HH\otimes A}=\{0\}$, and so the representation is automatically Cuntz-Pimsner covariant. On the other hand, if $\dim(\HH)<\infty$, then $\KK(\HH)=B(\HH)$, and so $J_{\HH\otimes A}=A$ since $\KK(\HH)\otimes A\simeq \KK(\HH\otimes A)$. After choosing an orthonormal basis for $\HH$, routine calculations show that the representation is Cuntz-Pimsner covariant. The induced homomorphism from $\OO_{\HH\otimes A}$ to $\OO_\HH\otimes A$ is clearly surjective, since its range contains the generators of $\OO_\HH\otimes A$. Tensoring the gauge action on $\OO_\HH$ with the identity map on $A$ gives a gauge action on $\OO_\HH\otimes A$ compatible with the representation, and so injectivity follows from an application of the gauge-invariant uniqueness theorem
\cite[Theorem 6.4]{katsuracorrespondence}
\end{proof}

\begin{rem}
We formulated \corref{central} to get the conclusion regarding $\OO_X$,
but since $X$ is isomorphic to the external tensor product of $\ell^2(G/H)$ and $A$
we could deduce other facts as well.
For example,
\[
\KK(X)
\simeq \KK\bigl(\ell^2(G/H)\otimes A\bigr)\simeq \KK(\ell^2(G/H))\otimes A.
\]
Since $\KK(X)\simeq C_0(G/H)\rtimes G$ by Rieffel's imprimitivity theorem,
we have a tensor-product decomposition of the crossed product:
\[
C_0(G/H)\rtimes G\simeq \KK(\ell^2(G/H))\otimes C^*(H).
\]
Of course, this  observation is not new; for example, since $G$ acts trivially on the open central subgroup $H$,
we could deduce this decomposition from \cite[Corollary~2.10]{gre:structure}.
\end{rem}

\begin{rem}
There is a unique continuous action $\alpha\:H\to \aut \OO_{\ell^2(G/H)}$ such that
\begin{equation}\label{act on On}
\alpha_h(S_{tH})=S_{htH}\righttext{for}h\in H,tH\in G/H.
\end{equation}
This is routine: $H$ acts continuously on the discrete space $G/H$,
giving a strongly continuous unitary representation of $H$ on the Hilbert space $\ell^2(G/H)$,
which by universal properties determines a continuous action of $H$ by automorphisms on $\OO_{\ell^2(G/H)}$.
\end{rem}

\section{$G$ discrete}\label{discrete}

Suppose that $G$ is discrete
and $H$ is any subgroup.
We identify a group element $s\in G$ with the characteristic function of $\{s\}$, so that
\[
c_c(G)=\spn G
\]
is a dense subspace of the $C^*(H)$-correspondence $X$.
Similarly, we have $c_c(H)=\spn H$, which is a dense *-subalgebra of $A=C^*(H)$.

In the discrete case we will modify our notation for Toeplitz representations of the $C^*(H)$-correspondence $X$: we use $U$ rather than $\pi$ for a representation of $A$, to remind us that it is the integrated form of a unitary representation of the discrete group $H$.

Choose a cross section $\eta\:G/H\to G$,
and
define $\kappa\:H\times G/H\to H$ by
\[
\kappa(h,tH):=\eta(htH)\inv h\eta(tH)
\]

\begin{lem}\label{cocycle}
With the above notation, $\kappa$ is a cocycle for the action of $H$ on $G/H$.
\end{lem}

\begin{proof}
This is just the canonical cocycle $G\times G/H\to H$ restricted to $H\times G/H$.
\end{proof}

For $s,t\in G$ we have
\[
\<s,t\>_A=\begin{cases}s\inv t\case sH=tH\\0\ifnot.\end{cases}
\]
Thus for $h\in H$,
\[
\theta_{s,s}sh=s\<s,sh\>_A=sh,
\]
while $\theta_{s,s}=0$ on $\spn\{t:t\notin sH\}$.

In the correspondence $X$, the set of representatives $\{\eta(tH):tH\in G/H\}$ is orthonormal,
and we have
\begin{equation}\label{identity}
h\eta(tH)=\eta(htH)\kappa(h,tH)\righttext{for}h\in H,tH\in G/H.
\end{equation}

Our analysis of $\OO_X$ will depend on whether the index $[G:H]$ is finite or infinite.
If $[G:H]<\infty$,
then $X$ is algebraically finitely generated, so $X$ is finitely generated projective
as a Hilbert $A$-module,
which simplifies things a great deal.
Rather than appeal to general theory, though, we show how this works in our special situation.
Because $G/H$ is finite, in
$\LL(X)$ we have
\begin{equation}\label{eq fin basis}
\sum_{tH\in G/H}\theta_{\eta(tH),\eta(tH)}=1.
\end{equation}
In particular, $\KK(X)=\LL(X)$.
Thus the correspondence $X$ is regular, i.e., $J_X=A$
--- of course we already knew this because $H$ is open and has finite index in $G$.
Also, $\OO_X$ is unital,
and for every Cuntz-Pimsner covariant representation $(\psi,U)$ of $X$
the associated homomorphism
$k_X^{(1)}$ of $\KK(X)$ is unital.

\begin{prop}\label{US prop}
Let
$H$ be a subgroup of a discrete group $G$,
and let
$B$ be a unital $C^*$-algebra.
Then
the Cuntz-Pimsner covariant representations of the $C^*(H)$-correspondence $X$ in $B$
are in 1-1 correspondence with pairs $(\Psi,U)$,
where $\Psi\:\OO_{\ell^2(G/H)}\to B$ is a unital homomorphism,
$U\:H\to B$ is a unitary homomorphism,
and
\begin{equation}\label{US}
U_h\Psi(S_{tH})=\Psi(S_{htH})U_{\kappa(h,tH)}\righttext{for}h\in H,tH\in G/H.
\end{equation}
\end{prop}

\begin{proof}
First suppose that $(\psi,U)$ is a Cuntz-Pimsner covariant representation of $X$ in $B$.
Then for $tH,uH\in G/H$ we have
\[
\psi(\eta(tH))^*\psi(\eta(uH))
=U_{\<\eta(tH),\eta(uH)\>_A}=\begin{cases}1\case tH=uH\\0\ifnot,\end{cases}
\]
since the set $\{\eta(tH):tH\in G/H\}$ is orthonormal in the Hilbert $A$-module $X$.
Thus the $\psi(\eta(tH))$ are isometries with mutually orthogonal ranges.

If $[G:H]<\infty$ then,
since the correspondence $X$ is regular and nondegenerate,
the homomorphism $\psi^{(1)}\:\KK(X)\to B$ is unital, so
\begin{align*}
\sum_{tH\in G/H}\psi(\eta(tH))\psi(\eta(tH))^*
&=\sum_{tH\in G/H}\psi^{(1)}(\theta_{\eta(tH),\eta(tH)})
=\psi^{(1)}(1)
=1.
\end{align*}

Thus in all cases
there is a unique unital homomorphism $\Psi\:\OO_{\ell^2(G/H)}\to B$ such that
\[
\Psi(S_{tH})=\psi(\eta(tH))\righttext{for}tH\in G/H.
\]

For
\eqref{US},
if $h\in H$ and $tH\in G/H$
then
by \eqref{identity}
\begin{align*}
U_h\Psi(S_{tH})
&=\psi(h\eta(tH))
=\psi\bigl(\eta(htH)\kappa(h,tH)\bigr)
=\Psi(S_{htH})U_{\kappa(h,tH)}.
\end{align*}

Now suppose that $(\Psi,U)$ is a pair as in the Proposition.
Since
the map $(tH,h)\mapsto \eta(tH)h$ from $G/H\times H$ to $G$ is bijective,
the set
\[
\{\eta(tH)h:tH\in G/H,h\in H\}
\]
is a linear basis for $c_c(G)$,
so
there is a unique linear map $\psi\:c_c(G)\to B$ such that
\[
\psi(\eta(tH)h)=\Psi(S_{tH})U_h.
\]
Since $X$ is the completion of the $c_c(H)$-precorrespondence $c_c(G)$,
the following computations imply that the pair $(\psi,U)$ is a Toeplitz representation of $X$ in $B$:
for $tH,uH\in G/H$ and $h,k\in H$,
\begin{align*}
\psi(\eta(tH)h)^*\psi(\eta(uH)k)
&=(\Psi(S_{tH})U_h)^*\Psi(S_{uH})U_k
\\&=U_h^*\Psi(S_{tH})^*\Psi(S_{uH})U_k,
\end{align*}
which,  since the $\Psi(S_{tH})$ are isometries with mutually orthogonal ranges and the representatives
$\{\eta(tH): tH\in G/H\}$ are orthonormal, equals
$U_{\<\eta(tH)h,\eta(tH)k\>_A}$.
Further,
\begin{align*}
U_h\psi(\eta(tH)k)
&=U_h\Psi(S_{tH})U_k
\\&=\Psi(S_{htH})U_{\kappa(h,tH)}U_k
\\&=\Psi(S_{htH})U_{\kappa(h,tH)k}
\\&=\psi\bigl(\eta(htH)\kappa(h,tH)k\bigr)
\\&=\psi(h\eta(tH)k).
\end{align*}
If $[G:H]=\infty$, the Toeplitz representation $(\psi,U)$ is automatically Cuntz-Pimsner covariant.
On the other hand, if
$[G:H]<\infty$, we must verify Cuntz-Pimsner covariance:
for $h\in H$,
since
\begin{align*}
\varphi(h)
&=\varphi(h)1
\\&=\sum_{tH\in G/H}\varphi(h)\theta_{\eta(tH),\eta(tH)}
\\&=\sum_{tH\in G/H}\theta_{h\eta(tH),\eta(tH)},
\end{align*}
we have
\begin{align*}
\psi^{(1)}(\varphi(h))
&=\sum_{tH\in G/H}\psi^{(1)}(\theta_{h\eta(tH),\eta(tH)})
\\&=\sum_{tH\in G/H}\psi(h\eta(tH))\psi(\eta(tH))^*
\\&=\sum_{tH\in G/H}U_h\Psi(S_{tH})\Psi(S_{tH})^*
\\&=U_h.
\end{align*}

Thus we have defined procedures going both ways:
starting with a Cuntz-Pimsner covariant representation $(\psi,U)$ of $X$ in $B$,
we produced a pair $(\Psi,U)$ as in the Proposition,
and on the other hand,
starting with a pair $(\Psi,U)$ as in the Proposition,
we produced a Cuntz-Pimsner covariant representation $(\psi,U)$ of $X$ in $B$.
We verify that these procedures are inverse to each other:
first, if we use $(\psi,U)$ to produce $(\Psi,U)$, and then in turn use that to produce $(\psi',U)$,
then for all $tH\in G/H,h\in H$ we have
\begin{align*}
\psi'(\eta(tH)h)
&=\Psi(S_{tH})U_h
\\&=\psi(\eta(tH))U_h
\\&=\psi(\eta(tH)h),
\end{align*}
and it follows that $\psi'=\psi$.
On the other hand, if we use $(\Psi,U)$ to produce $(\psi,U)$, and then in turn use that to produce $(\Psi',U)$,
then for all $tH\in G/H$ we have
\[
\Psi'(S_{tH})=\psi(\eta(tH))=\Psi(S_{tH}),
\]
and it follows that $\Psi'=\Psi$.
\end{proof}

\begin{rem}
If $[G:H]<\infty$, then the correspondence $X$ is nondegenerate and regular,
so \corref{ind res reg} applies, and hence the Cuntz-Pimsner covariant representations of $X$ on a Hilbert space $\HH$ are in 1-1 correspondence with the pairs $(U,V)$, where $U$ is a unitary representation of $H$ on $\HH$ and $V\:X\otimes_A \HH\to \HH$ implements a unitary equivalence between $X\dashind U$ and $U$.
Comparing with \propref{US prop} above, it makes sense to ask, given a pair $(\Psi,U)$, where $\Psi$ is a unital representation of $\OO_{\ell^2(G/H)}$ on $\HH$ and $U$ is a unitary representation of $H$ on $\HH$ satisfying \eqref{US}, what is the associated unitary intertwiner $V$?
Comparing the proofs of \lemref{ind res} and \propref{US prop},
it is easy to see that $V\:X\otimes_A \HH\to \HH$ is the unique bounded linear map such that
\[
V(\eta(tH)\otimes \xi)=\Psi(S_{tH})\xi\righttext{for}tH\in G/H,\xi\in\HH.
\]
However, it turns out that it would not save any time or effort to use \corref{ind res reg} to help prove \propref{US prop}.
\end{rem}

\begin{rem}
If $[G:H]<\infty$, then
\propref{US prop}
is closely related to (indeed, essentially a special case of) \cite[Discussion on page~298]{KajPinWatIdeal}.
To see this,
recall from \cite{KajPinWatIdeal}
that a finite set $\{y_1,\dots,y_n\}\subset X$ is called a \emph{basis} for $X$ if
$x=\sum_{i=1}^ny_i\<y_i,x\>_A$ for all $x\in X$,
and then for all $a\in A$ and all $j$ we have
\[
\varphi(a)y_j=\sum_{i=1}^ny_ia_{ij},
\]
where $a_{ij}=\<y_i,\varphi(a)y_j\>_A$.
\cite{KajPinWatIdeal} then shows that $\OO_X$ is the universal $C^*$-algebra generated by $A$ and $n$ elements $S_1,\dots,S_n$ satisfying
\begin{itemize}
\item $S_i^*S_j=\<y_i,y_j\>_A$,
\item $\sum_{i=1}^nS_iS_i^*=1$, and
\item $aS_j=\sum_{i=1}^nS_ia_{ij}$ for all $a\in A$ and $j=1,\dots n$.
\end{itemize}
In our setting, we have $A=C^*(H)$,
and we are
assuming that $H$ has finite index $n$ in $G$. Then \eqref{eq fin basis} shows that 
$\{\eta(uH)\}_{uH\in G/H}$ is a basis of the $C^*(H)$-correspondence $X$. By the discussion preceding \eqref{eq fin basis}, this 
basis is orthonormal.
Thus by \cite{KajPinWatIdeal}
$\OO_X$ is universally generated by
$A$ and
a Cuntz family of
isometries $\{S_{uH}\}_{uH\in G/H}$
satisfying
\[
hS_{tH}=\sum_{uH\in G/H}S_{uH}a_{uH,tH},
\]
where
\begin{align*}
a_{uH,tH}
&=\<\eta(uH),h\eta(tH)\>_A.
\end{align*}
Now,
\[
h\eta(tH)=\eta(htH)\kappa(h,tH),
\]
therefore
\begin{align*}
\<\eta(uH),h\eta(tH)\>_A
&=\<\eta(uH),\eta(htH)\kappa(h,tH)\>_A
\\&=\begin{cases}
\kappa(h,tH)\case uH=htH\\
0\case uH\ne htH,
\end{cases}
\end{align*}
and so the scheme of \cite{KajPinWatIdeal}
says that
\begin{align*}
hS_{tH}
&=S_{htH}\kappa(h,tH),
\end{align*}
which is the condition \eqref{US} of \propref{US prop}.
\end{rem}

\begin{rem}
Inspection of \eqref{US} shows that if the cocycle $\kappa\:H\times G/H\to H$ satisfies $\kappa(h,tH)=h$ for all $(h,tH)\in H\times G/H$, then
the Cuntz-Pimsner algebra $\OO_X$ is isomorphic to the crossed product $\OO_{G/H}\rtimes_\alpha H$, where $\alpha\:H\to \aut \OO_{G/H}$ is the action defined by \eqref{act on On}. The condition on $\kappa$ is satisfied, for example,  if  the cross section $\eta\:G/H\to G$ is equivariant for the left $H$-actions:
\[
h\eta(tH)=\eta(htH)\righttext{for}h\in H,tH\in G/H,
\]
which however forces $H=\{1\}$ since $H$ acts freely on $G$ but has a fixed point in $G/H$. The referee has kindly pointed out to us that the condition  is also satisfied when $G$ is abelian or when $H=G$ and $\eta(H)=1$.
\end{rem}

\begin{cor}\label{EP}
Let $G$ be discrete, let
$E$ be the directed graph with one vertex and edge set $E^1=G/H$,
and let $H$ act on $E$ by fixing the vertex and
acting on the edges by left translation on the homogeneous space.
Then
$\kappa$ is a cocycle for the action of $H$ on the graph $E$ in the sense of
\cite[Definition~3.3]{bkqexelpardo},
and the correspondence $X$ is isomorphic to the associated correspondence
$Y^\kappa$ of \cite[Definition~3.6]{bkqexelpardo},
and so the Cuntz-Pimsner algebra $\OO_X$ is isomorphic to the
Exel-Pardo algebra $\OO_{Y^\kappa}$ of \cite[Definition~3.8]{bkqexelpardo}.
If $H$ has finite index in $G$,
then the graph $E$ is finite, and so $X$
is isomorphic to the associated correspondence
$M$ of \cite[Section~10]{exelpardo},
and so $\OO_X$ is isomorphic to the algebra $\OO_{H,G/H}$ of \cite[Definition~3.2]{exelpardo}.
\end{cor}

\begin{proof}
Recall from
\cite[Definition~3.6]{bkqexelpardo}
that the correspondence $Y^\kappa$ is constructed as follows:
first of all, since $E$ has only one vertex we can identify $c_0(E^0)\rtimes H$ with $A=C^*(H)$.
Now
give the set $G/H\times H$
the following operations,
for $tH,uH\in G/H,h,k\in H$:
\begin{itemize}
\item $(tH,k)h=(tH,kh)$;

\item $\<(tH,h),(uH,k)\>_A=\begin{cases}h\inv k\case tH=uH\\0\ifnot;\end{cases}$

\item $h(tH,k)=(htH,\kappa(h,tH)k)$.
\end{itemize}
Then the linear span $c_c(G/H\times H)$ becomes a $c_c(H)$-precorrespondence,
whose completion is $Y^\kappa$.
It is routine to check that the map
\[
(tH,h)\mapsto \eta(tH)h\:G/H\times H\to G
\]
integrates to an isomorphism $Y^\kappa\simeq X$ as $C^*(H)$-correspondences.
\end{proof}

\begin{rem}\label{self}
Since the graph $E$ described
in \corref{EP}
has only one vertex, we are actually in the situation of a self-similar group action,
so $\OO_X$ is isomorphic to the $C^*$-algebra $\OO_{(H,G/H)}$ of \cite[Definition~3.1]{nek}
(see also \cite[Proposition~3.2 and Remark~3.6]{lrrw} or \cite[Example 3.3]{exelpardo}).
\end{rem}

\begin{rem}\label{cohomology}
The Cuntz-Pimsner algebra $\OO_X$ does not have anything directly to do with the cross section $\eta$,
but obviously the Exel-Pardo correspondence $Y^\kappa$ does.
So \corref{EP} raises an obvious issue: how is the independence of $\OO_X$ upon $\eta$ reflected in $\OO_{Y^\kappa}$? More precisely, if we choose another cross section $\eta'\:G/H\to G$,
and use it to define another cocycle $\kappa'\:H\times G/H\to H$,
then clearly the Exel-Pardo algebras $\OO_{Y^\kappa}$ and $\OO_{Y^{\kappa'}}$ must be isomorphic,
since they are both isomorphic to $\OO_X$; could we have predicted this just using the theory of cocycles?
The answer is
yes, because the cocycles $\kappa$ and $\kappa'$ will be cohomologous.
For completeness, we include a reminder:
Two cocycles $\kappa,\kappa'$ for the action of $H$ on $G/H$ are called \emph{cohomologous}
if there is a map $\nu\:G/H\to H$ such that
\[
\kappa'(h,tH)=\nu(htH)\inv \kappa(h,tH)\nu(tH)\righttext{for}h\in H,tH\in G/H.
\]
Let $\kappa$ be defined using the cross section $\eta\:G/H\to G$ as above.
Given a map $\nu:G/H\to H$,
we get another cross section
\[
\eta'(tH)=\eta(tH)\nu(tH),
\]
and conversely, given another cross section $\eta'\:G/H\to G$,
we get a map $\nu\:G/H\to H$ defined by
\[
\nu(tH)=\eta(tH)\inv \eta'(tH),
\]
and it is well-known that the two cocycles associated to the cross sections $\eta,\eta'$ are cohomologous:
\begin{align*}
\kappa'(h,tH)
&=\eta'(htH)\inv h\eta'(tH)
\\&=\bigl(\eta(htH)\nu(htH)\bigr)\inv h\eta(tH)\nu(tH)
\\&=\nu(htH)\inv \eta(htH)\inv h\eta(tH)\nu(tH)
\\&=\nu(htH)\inv \kappa(h,tH)\nu(tH).
\end{align*}
It then follows that the two correspondences $Y^{\kappa}$ and $Y^{\kappa'}$,
and hence the associated Exel-Pardo algebras $\OO_{Y^\kappa}$ and $\OO_{Y^{\kappa'}}$,
are isomorphic \cite[Theorem~4.8]{bkqexelpardo}.

It might be of interest to interpret the above in terms of a classification result of Zimmer:
the orbits of the action of $H$ on $G/H$ are the double cosets in $H\under G/H$.
Thus the cocycle $\kappa$ is uniquely determined by the restricted cocycles
$\kappa|_{H\times HtH}$.
For each coset $tH\in G/H$
the stability subgroup of the action of $H$ is
\[
H_{tH}:=H\cap \eta(tH)H\eta(tH)\inv=\{h\in H:htH=tH\}.
\]
Then the action on the orbit $HtH$ is conjugate to the action of $H$ on the coset space $H/H_{tH}$,
and a result of Zimmer \cite[4.2.13]{Zimmer}
(also recorded in a form more convenient for our purposes in \cite[Lemma~2.8]{bkqexelpardo})
classifies those:
the cohomology classes of such cocycles are in 1-1 correspondence with
the set of conjugacy classes of homomorphisms from $H_{tH}$ to $H$.
The restricted cocycle
\[
\kappa_{tH}\:H\times H/H_{tH}\to H
\]
is given by
\[
\kappa_{tH}(h,k H_{tH})=\kappa(h,k tH)\righttext{for}h,k\in H.
\]
The homomorphism $\tau_{tH}\:H_{tH}\to H$ associated with the
restricted cocycle $\kappa_{tH}$ is given by
\[
\tau_{tH}(h)=\kappa_{tH}(h,H_{tH})=\kappa(h,tH)\righttext{for}h\in H_{tH}.
\]
Conversely, starting with a homomorphism $\tau\:H_{tH}\to H$,
the associated cocycle $\mu\:H\times H/H_{tH}\to H$ is constructed by
first choosing a cross section $\gamma\:H/H_{tH}\to H$ with $\gamma(H_{tH})=1$,
and then defining
\[
\mu(h,k H_{tH})=\gamma(hk H_{tH})\inv h\gamma(k H_{tH}).
\]
In the case of the Rieffel $A$-correspondence $X$,
the unique cohomology class of cocycles
is determined by the inclusion homomorphisms $H_{tH}\hookrightarrow H$
for each $tH\in G/H$.
\end{rem}

\begin{q}
\corref{EP} leads to another obvious question: what Exel-Pardo algebras arise in this manner?
Put another way, what cocycles $\kappa$ arise from the above procedure?
More precisely,
if we start with an action of $H$ on a set $T$
and a cocycle $\kappa\:H\times T\to H$ for this action,
when will there exist a group $G$ containing $H$ as a subgroup such that $G/H$ can be identified with $T$ and $\kappa$ arises as above?
There is one obvious obstruction: there must be at least one fixed point in $T$, since $H$ fixes the coset $H$ in $G/H$.
Are there any other obstructions?
For example, can we realize all of Katsura's algebras $\OO_{A,B}$
\cite{KatKirchberg} (also see \cite[Example~3.4]{exelpardo}),
which include all Kirchberg algebras in the UCT class?

Another obstruction is the cohomology class of the cocycle:
as we mentioned in \remref{cohomology},
for every double coset $HtH$ the cohomology class of the restricted cocycle corresponds to the inclusion homomorphism $H_{tH}\hookrightarrow H$.
Thus it would appear that we do not get all cocycles.
\end{q}

\section{$H$ compact}\label{compact}

In this section we show that the Cuntz-Pimsner algebra arising from a compact subgroup is Morita equivalent, and often  isomorphic, to a graph algebra.
Recall that we are assuming that our group $G$ is second countable, so that the $C^*(H)$-correspondence $X$ is separable.
First we need some preliminaries. Recall from \cite[Section~4.1.1 and Addendum~4.7.20(iv)]{dixmier} that a $C^*$-algebra is called \emph{elementary} if it is isomorphic to the algebra of compact operators on a Hilbert space, and \emph{dual} if it is a $c_0$-direct sum \
\[
A=\bigoplus_{\mu\in\Omega}A_\mu
\]
of elementary algebras.
We can identify the spectrum $\what A$ of $A$ with the set $\Omega$.
Any two dual algebras with spectrum $\Omega$ are Morita equivalent,
and we need a particular consequence regarding Cuntz-Pimsner algebras.
In keeping with our blanket separability hypotheses, we assume that $\Omega$ is countable and that every $A_\mu$ is separable.

Let $A$ and $B$ be dual algebras with spectrum $\Omega$,
and with component elementary algebras $A_\mu$ and $B_\mu$.
For each $\mu\in\Omega$ choose an $A_\mu-B_\mu$ imprimitivity bimodule $M_\mu$, and define an $A-B$ imprimitivity bimodule $M$ by
\[
M=\bigoplus_{\mu\in\Omega}M_\mu.
\]
Let $X$ be a faithful nondegenerate $A$-correspondence,
and define a faithful nondegenerate $B$-correspondence $Y$ by
\[
Y=M^*\otimes_AX\otimes_AM.
\]
Then $X$ and $Y$ are Morita equivalent correspondences in the sense of \cite[Definition 2.1]{muhsolmorita}, and hence by \cite[Theorem~3.5]{muhsolmorita} the Cuntz-Pimsner algebras $\OO_X$ and $\OO_Y$ are Morita equivalent.
Note that, since $A$ and $B$ are separable by assumption, so is $M$, and hence so is $Y$.

In the particular case where all the $B_\mu$ are 1-dimensional, so that $B$ is commutative,
by \cite[Theorem~1.1]{graphcorres} $Y$ is isomorphic to the correspondence associated to a directed graph $E$ with vertex set $\Omega$
and in which for $\mu,\nu\in\Omega$ the cardinality of $\mu E^1\nu$
is the dimension of the Hilbert space $p_\mu Yp_\nu$,
where $p_\mu$ denotes the identity element of $B_\mu$,
regarded as a central projection in $B$.
Thus $\OO_Y\simeq C^*(E)$,
and hence $\OO_X$ is Morita equivalent to the graph algebra $C^*(E)$.
For this to be useful, we would like to be able to find the edges of the graph $E$ directly using the $A$-correspondence $X$.
For each $\mu\in\Omega$ choose associated irreducible representations $\pi_\mu$ of $A$ and $\tau_\mu$ of $B$.
Then
by the construction of $E$ in \cite{graphcorres},
the cardinality of $\mu E^1\nu$ coincides with the multiplicity of $\tau_\mu$ in the induced representation $Y\dashind \tau_\nu$.
Thus
we expect the following:

\begin{lem}\label{multiplicity}
For all $\mu,\nu\in\Omega$,
the cardinality of $\mu E^1\nu$
equals the multiplicity of $\pi_\mu$ in the representation $X\dashind \pi_\nu$.
\end{lem}

\begin{proof}
It suffices to show that
for all $\mu,\nu\in\Omega$
the multiplicity of $\pi_\mu$ in $X\dashind \pi_\nu$ equals the multiplicity of $\tau_\mu$ in $Y\dashind\tau_\nu$.
This is almost obvious, and we include the routine computation.
By \cite[Theorem~3.29]{tfb},
we have unitary equivalences
\[
M\dashind\tau_\mu\simeq \pi_\mu\righttext{for all}\mu\in\Omega.
\]
Fix $\nu\in\Omega$, and
decompose $X\dashind \pi_\nu$ into irreducibles:
\begin{align*}
X\dashind\pi_\nu\simeq \bigoplus_{\mu\in\Omega}n_\mu\pi_\mu,
\end{align*}
where $n_\mu$ is the multiplicity of $\pi_\mu$ in $X\dashind\pi_\nu$.
Then we have
\begin{align*}
Y\dashind\tau_\nu
&\simeq M^*\dashind X\dashind M\dashind \tau_\nu
\\&\simeq M^*\dashind X\dashind \pi_\nu
\\&\simeq M^*\dashind \bigoplus_{\mu\in\Omega}n_\mu\pi_\mu
\\&\simeq \bigoplus_{\mu\in\Omega}M^*\dashind n_\mu\pi_\mu
\righttext{(by \cite[Proposition~2.69]{tfb})}
\\&\simeq \bigoplus_{\mu\in\Omega}n_\mu M^*\dashind \pi_\mu
\\&\simeq \bigoplus_{\mu\in\Omega}n_\mu \tau_\mu,
\end{align*}
and the result follows.
\end{proof}

\begin{cor}\label{dual CP}
When
$X$ is a separable faithful nondegenerate correspondence over a separable
dual $C^*$-algebra $A$,
the Cuntz-Pimsner algebra $\OO_X$ is Morita equivalent to the graph algebra $C^*(E)$ of a directed graph with vertex set $\what A$ and in which, for all $\pi,\sigma\in\what A$,
the number of edges from $\sigma$ to $\pi$
is the multiplicity of $\pi$ in $X\dashind\sigma$.
If $A$ is commutative then $\OO_X\simeq C^*(E)$.
\end{cor}

Now let $H$ be a compact subgroup of our second countable group $G$,
let $A=C^*(H)$, and let $X$ be the $A$-correspondence
for Rieffel induction.
Note that we can identify the spectrum of $C^*(H)$ with the set $\what H$ of irreducible unitary representations of $H$
(see \cite[Remark~2.41]{danacrossed}).
Then $A$ is a dual algebra
by \cite[Proposition~3.4]{danacrossed},
so by the above we have:

\begin{cor}\label{compact CP}
When $H$ is compact the Cuntz-Pimsner algebra $\OO_X$ is Morita equivalent to the graph algebra $C^*(E)$ of a directed graph with vertex set $\what H$ and in which, for all $U,V\in\what H$,
the number of edges from $V$ to $U$
is the multiplicity of $U$ in $X\dashind V$.
If $H$ is abelian then $\OO_X\simeq C^*(E)$.
\end{cor}

\begin{q}
Which directed graphs arise as in \corref{compact CP}?
It follows from \cite[Corollary~3 of Theorem~5.5]{fel:weak2} that any such graph has at least one loop edge at every vertex.
\end{q}

\begin{rem}\label{all correspondences}
One could push the above machinery further, to
classify up to isomorphism
all faithful nondegenerate $A$-correspondences,
where $A=\bigoplus_{\mu\in\Omega}A_\mu$ is a countable direct sum of separable elementary $C^*$-algebras $A_\mu$,
but since we do not need this for our results we only give a very rough outline.
As above, let $B=\bigoplus_{\mu\in\Omega}B_\mu$ be a commutative $C^*$-algebra with spectrum $\Omega$.
For each $\mu\in\Omega$ there is up to isomorphism a unique $A_\mu-B_\mu$ imprimitivity bimodule $M_\mu$, namely any Hilbert space of the appropriate dimension, and as before let $M=\bigoplus_{\mu\in\Omega}M_u$ be the associated $A-B$ imprimitivity bimodule.
Every faithful nondegenerate $A$-correspondence $X$
gives rise to a faithful nondegenerate $B$-correspondence
$Y=M^*\otimes_AX\otimes_AM$,
and this process is reversible:
\[
M\otimes_BM^*\otimes_AX\otimes_AM\otimes_BM^*
\simeq A\otimes_AX\otimes_AA
\simeq X,
\]
since $AX=X$.
The $B$-correspondence $Y$
is characterized up to isomorphism by the directed graph $E$ with vertex set $\Omega$ and the number of edges from $\nu$ to $\mu$ given by the dimension of the Hilbert space $p_\mu Yp_\nu$,
where $p_\mu$ is the identity element of $B_\mu$, regarded as a minimal projection in $B$.
Up to isomorphism, the $A$-correspondence $X$ can be decomposed as
\[
\bigoplus_{\mu,\nu\in\Omega}M_\mu^*\otimes_{B_\mu}p_\mu Yp_\nu\otimes_{B_\nu}M_\nu^*,
\]
which depends only upon the dimensions of the Hilbert spaces $p_\mu Yp_\nu$.
\end{rem}

\section{Examples}\label{ex}

Interesting examples arise already with finite groups.
So, let $H$ be a subgroup of a finite group $G$.
Since $H$ is finite, it is compact, so by \corref{compact CP} the Cuntz-Pimsner algebra $\OO_X$ is Morita equivalent to the $C^*$-algebra of a directed graph $E$ with $E^0=\what H$
and in which, for $U,V\in \what H$,
the cardinality of $U E^1 V$ is the multiplicity of $U$ in $X\dashind V$.
To compute these multiplicities, we appeal to Mackey's Subgroup Theorem
\cite[Theorem~7.1]{mackeyinduced},
which in our situation can be expressed in the form
\[
X\dashind V\simeq \bigoplus_{HsH\in H\under G/H}\ind_{H_s}^H V^s,
\]
where
\[
H_s=H\cap s\inv Hs
\midtext{and}
V^s=V\circ \ad s|_{H_s},
\]
and where in the direct sum we take one representative $s$ from each double coset $HsH$.
Note that $H\under G/H$ is finite since $G$ is.

As we observed in \secref{subgroups}, the cases $H=\{1\}$ or $H=G$ are boring,
so we focus on proper nontrivial subgroups.
The case $H=\Z_2=\Z/2\Z$ is already interesting, so we examine it in some detail.
First note that, since $\Z_2$ is abelian, by \corref{compact CP} we actually have $\OO_X\simeq C^*(E)$ for the above directed graph $E$.

If the subgroup $H=\Z_2$ is normal, then it is central (and open, since $G$ is finite), so by \corref{central} we have
$\OO_X\simeq \OO_{[G:H]}\otimes \C^2$.
So we assume from now on that $H$ is nonnormal.
Then the action of $H$ on $G/H$ has at least one fixed point (namely $H$) and at least one 2-element orbit.
Let
\begin{align*}
&\text{$r$ be the number of fixed points in $G/H$, and}
\\
&\text{$q$ the number of 2-element orbits.}
\end{align*}
Note that $r$ is the index
$[Z_G(H):H]$
of $H$ in its centralizer $Z_G(H)$,
and $[G:H]=r+2q$.

What pairs $(r,q)$ can occur?

\begin{prop}\label{which pairs}
With the above notation, a pair $(r,q)$ of positive integers can arise if and only if $r\mid 2q$.
\end{prop}

\begin{proof}
First suppose that $H$ is a proper nonnormal subgroup of a finite group $G$ with $H\simeq \Z_2$.
As above, put $r=[Z_G(H):H]$,
and let $[G:H]=r+2q$, so that $q$ is a positive integer.
We have
$|G|=2r+4q$.
Also,
\[
|Z_G(H)|=2r,
\]
which must divide $|G|$,
i.e.,
$2r\mid (2r+4q)$.
Thus
$2r\mid 4q$,
so
$r\mid 2q$.

Conversely, let $r$ and $q$ be positive integers with $r\mid 2q$, say $2q=mr$.
We must show that there exists a finite group $G$
containing a subgroup $H\simeq \Z_2$
such that
\[
[Z_G(H):H]=r\midtext{and}[G:H]=r+2q.
\]

Case 1.
$m$ is even.
Put
\[
G=\Z_r\times (\Z_{m+1}\rtimes \Z_2),
\]
where $H=\Z_2$ acts on $\Z_{m+1}$ by the automorphism $n\mapsto -n$. Since $m+1$ is odd, this automorphism has no fixed points other than the identity element 0,
so
\[
Z_G(H)=\Z_r\times \Z_2,
\]
and hence $[Z_G(H):H]=r$.
Further,
\[
[G:H]=r(m+1)=r+rm=r+2q.
\]

Case 2.
$m$ is odd.
Then $r$ is even, say $r=2j$.
Put
\[
G=\Z_j\times (\Z_{2(m+1)}\rtimes \Z_2),
\]
where again $H=\Z_2$ acts on $\Z_{2(m+1)}$ by $n\mapsto -n$.
In this case, the fixed-point subgroup under this action is $\{0,m+1\}$,
so
\[
Z_G(H)=\Z_j\times \{0,m+1\}\times \Z_2,
\]
and hence $[Z_G(H):H]=j\cdot 2=r$.
Further,
\[
[G:H]=j\cdot 2(m+1)=r(m+1)=r+2q,
\]
as desired.
\end{proof}

Now we continue the investigation of the directed graph $E$, begun in the first paragraph of this section.
Note that
for each $s\in Z_G(H)$ we have $H_s=H$ and $V^s=V$,
so the fixed points in $G/H$ contribute a summand
$rV$
in $X\dashind V$.

Each 2-element orbit in $G/H$ is a disjoint union
$sH\sqcup hsH$,
where $s\notin Z_G(H)$ and $h$ is the generator of $H$.
We have
$H_s=\{1\}$ and consequently $V^s$ is
(equivalent to)
the trivial character 1,
and so
\[
\ind_{H_s}^H V^s=\ind_{\{1\}}^H 1,
\]
which is
the regular representation $\lambda_H$ of $H$.
As $H$ is a finite abelian group, we have
\[
\lambda_H\simeq \bigoplus_{U\in \what H}U.
\]
Combining, we see that for each $V\in \what H$,
\[
X\dashind V\simeq rV\oplus q\bigoplus_{U\in \what H}U
=(r+q)V\oplus qU,
\]
where $U$ is the character of $H$ different from $V$.
Consequently, the associated graph $E$ has the form
\[
\xymatrix{
U \ar@(ul,dl)_{(r+q)} \ar@/_/[r]_{(q)}
&V \ar@(ur,dr)^{(r+q)} \ar@/_/[l]_{(q)}
}
\]
where $U,V$ are the two characters of $H$,
and where
a number in parentheses indicates the number of edges from the first vertex to the second.
Because we are assuming that $H$ is a proper nonnormal subgroup, we have $q>0$.
Thus the graph $E$ is finite and
transitive (meaning that $vE^1w\ne \varnothing$ for all $v,w\in E^0$),
and every cycle has an entry,
so by
\cite[Corollary~3.11]{kpr}
$C^*(E)$ is unital, simple, and purely infinite.
By \cite[Remark~4.3]{raeburngraph}, $C^*(E)$ is nuclear and in the bootstrap class.
Thus $\OO_X$, being
isomorphic to $C^*(E)$,
is classifiable up to Morita equivalence by its $K$-theory,
according to the classification theorem of Kirchberg and Phillips \cite{KirPhi, phillipsclassification}.
In fact, since $\OO_X$ is unital, it is classifiable up to isomorphism by $K_0$, $K_1$, and the class $[1]_0$ in $K_0$ of the identity $1_{\OO_X}$.


To compute the $K$-theory, by \cite[Theorem~7.16]{raeburngraph}
we can use the vertex matrix $A$, indexed by $E^0$, where the $ij$-entry is the number of edges from the $j$th vertex to the $i$th one.
And then the algorithm tells us
that, identifying the matrix $B:=A^t-1$ with an endomorphism of the free abelian group $\Z^{E^0}$, we have
\begin{align*}
K_1(C^*(E))&\simeq \ker B
\\
K_0(C^*(E))&\simeq \coker B=\Z^{E^0}/B\Z^{E^0},
\end{align*}
where the isomorphism for $K_0$ is given by sending the class of the vertex projection $[p_v]_0$ to $1_v+B\Z^{E^0}$.
(The usual formulation involves $1-A^t$, but in our case the matrix $A^t-1$ is more convenient, and the results are the same.)
In our situation we have $E^0=\{U,V\}$, and
\[
A=\mtx{r+q&q\\q&r+q},
\]
and so
\[
B=A^t-1=\mtx{r+q-1&q\\q&r+q-1}.
\]
Let $p=r+q-1$. Then $p\ge q>0$,
and we have
\begin{align*}
K_1&=\ker B
\\
K_0&=\Z^2/B\Z^2.
\end{align*}

Since
the graph algebra $C^*(E)$ is unital,
we must compute the class $[1]_0$ in $K_0$ of the identity $1_{C^*(E)}$.
For this, we need to compute the classes
$[p_v]_0$
of the vertex projections and then add them up.
In our case, we have
\[
[1]_0
=\mtx{1\\1}+B\Z^2.
\]

To compute the cokernel of $B$,
we appeal to the standard theory which identifies it with a direct sum of abelian groups
via computing  the Smith normal form of $B$. We recall from eg. \cite[Section~3.22]{SurveyMatrix} how this works. Let $B\in M_n(\Z)$, and suppose $B$ has rank $k$ for some $k\leq n$. For each $j=1,\dots ,k$, if $B$ has at least one nonzero $j$-square subdeterminant, define $f_j$ as the greatest common divisor of all $j$-th order subdeterminants of $B$. Set $f_0=1$. Note that $f_{j-1}$ divides $f_j$ for all $j=1,\dots, k$. The Smith normal form of $B$ is the diagonal matrix $N$ with diagonal entries
\[
q_j:=f_j/f_{j-1}\righttext{for}j=1,\dots, k.
\]
Note that $q_j$ divides $q_{j+1}$ for all $j=1,\dots, k$. Moreover, there are invertible matrices $C,D\in M_n(\Z)$ such that $B=CND$ and the map $x\to C^{-1}x$ on $\Z^n$ induces an isomorphism
\begin{equation*}
\Phi\:\Z^n/B\Z^n\to \Z^n/N\Z^n.
\end{equation*}
To compute the class of the identity in $K_0$ we will compute its image in $\Z^n/N\Z^n$ under $\Phi$. To get explicit formulas, we split up the analysis into several cases.

Case 1.
$r=1$.
Then $B=\smtx{q&q\\q&q}$.
Thus
$K_1=\ker B$ is the cyclic subgroup of $\Z^2$ generated by $\smtx{1\\-1}$,
so $K_1\simeq \Z$.

Clearly, $B$ has rank $1$. In the above notation we have $B=CND$ for
\[
C=\mtx{1&0\\1&1},\quad N=\mtx{q&0\\0&0},\quad D=\mtx{1&1\\0&1}.
\]
Denote by $(m,n)$ the transpose of a column-vector $\smtx{m\\n}$ in $\Z^2$. The map $(m,n)\mapsto (m(\operatorname{mod}\, q),n)$
 has kernel $N\Z^2$, and so induces an  isomorphism
$$
\Psi:\Z^2/N\Z^2\rightarrow \Z_q\oplus \Z.
$$

Composing $\Psi$ with $\Phi$ gives an isomorphism
\[
K_0\simeq \Z_q\oplus \Z,
\]
and since $C^{-1}$ carries $\smtx{1\\1}$ to $\smtx{1\\0}$ we have, in $\Z_q\oplus \Z$,
\[
[1]_0=(1,0).
\]

By \cite[Theorem~4.8 (3)]{spielbergBaumslag},
the
$C^*$-algebra of the category of paths
given by the positive submonoid $\Lambda$
of the
Baumslag-Solitar group
\[
BS(1,q+1)=\<a,b\mid ab=b^{q+1}a\>
\]
is UCT Kirchberg (by \cite[Corollary~4.10]{spielbergBaumslag}) and has $K$-theory
$(\Z_q\oplus \Z,\Z)$,
with $[1]_0=(1,0)$,
and hence when $r=1$ we have $\OO_X\simeq C^*(\Lambda)$.

\begin{ex}
Here is one of the simplest examples of the above:
let $H=\Z_2$ as a subgroup of the group $G=S_3$ of permutations of a 3-element set,
and let $X$ be the associated $C^*(H)$-correspondence.
It follows from the above analysis that
$\OO_X$ is isomorphic to the algebra of the following graph:
\bigskip
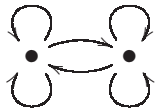
\[
\xymatrix{
\bullet \ar@/^/[r] \ar@(ur,ul)[] \ar@(dr,dl)[]
&
\bullet \ar@/^/[l] \ar@(ul,ur)[] \ar@(dl,dr)[]
}
\]
\bigskip

With the above notation, we have $r=q=1$,
so the $K$-groups of $\OO_X$ are
both $\Z$.
The crossed product of $PSL(2,\Z)$ acting on the boundary of the upper half plane,
and the Ruelle algebra associated to the 2-adic solenoid
are purely infinite simple $C^*$-algebras
with this $K$-theory
\cite[Application~15]{LSboundary},
so $\OO_X$ is Morita equivalent to both of these.
\end{ex}

Case 2. $r>1$. We have
\[
B=\mtx{p&q\\q&p},
\]
where $p>q>0$. Then
\[
K_1=\ker B=0.
\]

We turn to computing $K_0$. Since $B$ has rank two, we find that $f_0=1$, $f_1=\operatorname{gcd}(p,q)$ and $f_2=\operatorname{det} B=p^2-q^2$. Denote $d=p^2-q^2$.

We first suppose that $p$ and $q$ are coprime, so that $f_1=1$. The Euclidean algorithm gives
$s,t\in\Z$
such that
\[
sp+tq=1.
\]
The Smith normal form of $B$ and the associated invertible matrices $C,D$ are given as follows
\[
C=\mtx{p&-t\\q&s},\quad N=\mtx{1&0\\0&d},\quad D=\mtx{1&tp+sq\\0&1}.
\]
The map $(m,n)\mapsto (0, n(\operatorname{mod}\, d))$ induces an isomorphism
\[
\Psi\:\Z^2/N\Z^2\to \Z_1\oplus \Z_d,
\]
and the composition $\Psi\circ \Phi$ gives an isomorphism
\[
K_0\simeq \Z_1\oplus \Z_d=\Z_d.
\]
Since the isomorphism $\Phi$ carries the class of the identity in $K_0$ into
\[
C^{-1}\mtx{1\\1}=\mtx{s&t\\-q&p}\mtx{1\\1}=\mtx{s+t\\p-q},
\]
it follows that the image of $[1]_0$ in $\Z_d$ is identified as
\[
[1]_0=p-q.
\]
Now,
\[
d=p^2-q^2=(p-q)(p+q),
\]
so $[1]_0$ divides the order of the cyclic group $K_0$.
If $p-q=1$,
then it is a generator of $K_0$,
and so
\[
\OO_X\simeq \OO_{d+1}.
\]
On the other hand, if $p-q>1$, then
\[
\OO_X\simeq M_{p-q}(\C)\otimes \OO_{d+1}.
\]

Now suppose that $p$ and $q$ are not coprime. Rename $a=f_1=\gcd(p,q)$ and note that the
Smith normal form of $B$ is the matrix
\[
N=\mtx{a&0\\0&d/a}.
\]

Write $p=au$ and $q=av$. Then $u$ and $v$ are coprime, and $d=a^2g$, where $g=u^2-v^2$
is the determinant of the matrix
\[
B_1=\mtx{u&v\\v&u}.
\]
We then have $B=aB_1$, where the analysis of the coprime case applies to $B_1$. Choosing
$z,w\in \Z$ such that $zu+wv=1$, the matrix that plays the role of $C$ is now
\[
C_1=\mtx{u&-w\\v&z}.
\]
By the coprime case, we get an isomorphism
\[
K_0\simeq \Z_a\oplus \Z_{ag}.
\]
For the class of the identity in $K_0$,
in $\Z_a\oplus \Z_{ag}$ we have
\[
[1]_0=(z+w,u-v),
\]
where we can find suitable $z,w$
using either $zu+wv=1$ or $zp+wq=a$.

\begin{ex}\label{r 2}
If $r=2$, then $p=q+1$ is coprime to $q$.
We have
$p-q=1$,
$d=p+q=2q+1$,
$K_0\simeq \Z_{2q+1}$,
$[1]_0=1$,
and
\[
\OO_X\simeq \OO_{2q+2}.
\]
\end{ex}

\begin{ex}
If $q=1$, then $r=2$ (since we must have $r\mid 2q$ and we are assuming that $r>1$),
so this is a special case of \exref{r 2}:
we have $p=2$,
giving
\[
B=\mtx{2&1\\1&2},
\]
$d=3$,
$K_0\simeq \Z_3$,
$[1]_0=1$,
and
\[
\OO_X\simeq \OO_4.
\]
\end{ex}

\begin{ex}
If $r=q=2$,
then again we are in a special case of \exref{r 2},
and this time
$p=3$,
giving
\[
B=\mtx{3&2\\2&3},
\]
$d=5$,
$K_0\simeq \Z_5$,
$[1]_0=1$,
and
\[
\OO_X\simeq \OO_6.
\]
\end{ex}

\begin{ex}\label{r q}
If $r=q>2$, then $p=2q-1$,
which is coprime to $q$ since
\[
-p+2q=1.
\]
Thus
$p-q=q-1$,
$d=(q-1)(3q-1)$,
$K_0\simeq \Z_{(q-1)(3q-1)}$,
$[1]_0=q-1$,
and
\[
\OO_X\simeq M_{q-1}(\C)\otimes \OO_{(q-1)(3q-1)+1}.
\]
\end{ex}

\begin{ex}
As a special case of \exref{r q},
if $r=q=3$,
then
$p=5$,
$p-q=2$,
$d=16$,
$K_0\simeq \Z_{16}$,
$[1]_0=2$,
and
\[
\OO_X\simeq M_2(\C)\otimes \OO_{17}.
\]
\end{ex}

\begin{ex}
If $r=3$ and $q=6$,
then $p=8$ is not coprime to $q$.
We have
$a=\gcd(p,q)=2$,
$p-q=2$,
$u=\frac pa=4$, $v=\frac qa=3$,
$g=u^2-v^2=7$,
and
$K_0\simeq \Z_2\oplus \Z_{14}$.
Since
\[
8-6=2,
\]
we can take $z=1$ and $w=-1$, so
\[
[1]_0=(z+w,u-v)=(0,1).
\]
Note that if $r=3$ then,
since $r\mid 2q$ by \propref{which pairs}, we must have $3\mid q$, so in some sense this is the next biggest example after the preceding one, and the smallest one with $p$ and $q$ not coprime.
\end{ex}

\begin{ex}
If $r=3$ and $q=9$,
then $p=11$ is coprime to $q$.
We have
$p-q=2$,
$d=40$,
$K_0\simeq \Z_{40}$,
$[1]_0=2$,
and
\[
\OO_X\simeq M_2(\C)\otimes \OO_{41}.
\]
\end{ex}

\begin{ex}
If $r=6$ and $q=9$,
then $p=14$ is coprime to $q$,
and we have
$p-q=5$,
$d=115$,
$K_0\simeq \Z_{115}$,
$[1]_0=5$,
and
\[
\OO_X\simeq M_5(\C)\otimes \OO_{116}.
\]
\end{ex}

\section{Connection with \cite{ManRaeSut}}\label{MRS}

If the subgroup $H$ is a compact Lie group, then we can choose a faithful finite-dimensional unitary representation $\rho$. In this situation, \cite{ManRaeSut, ManRaeSut2} study the Doplicher-Roberts algebra $\OO_\rho$,
and show that it is Morita equivalent to a Cuntz-Krieger algebra ---
equivalently, a graph algebra, although
at the time \cite{ManRaeSut, ManRaeSut2} were written, the technology of graph $C^*$-algebras had not yet appeared.

The finite-dimensional Hilbert space $\HH$ of the representation $\rho$
can be regarded as an $A-\C$ correspondence, where $A=C^*(H)$ as before,
but there does not appear to be a natural way to give $\HH$ the structure of an $A$-correspondence.
Nevertheless, something interesting happens:
the method that \cite[Section~1]{ManRaeSut} use to construct a graph $E$ from $\rho$
is strikingly similar to our construction in \lemref{multiplicity}.
In \cite{ManRaeSut} the construction is as follows:
let $R$ be the set of equivalence classes of irreducible representations of $H$ occurring in the various tensor powers $\rho^{\otimes n}$; if $H$ is finite then $R=\what H$.
The graph $E$ has vertex set $R$,
and for each $\pi_1,\pi_2\in R$ the number of edges in $E$ from $\pi_2$ to $\pi_1$
is the multiplicity of $\pi_2$ in $\pi_1\otimes\rho$,
whereas in our \lemref{multiplicity} we define $E^0=\what H$,
and the number of edges from $\pi_2$ to $\pi_1$
is the multiplicity of $\pi_1$ in $X\dashind \pi_2$.
The similarity is uncanny,
particularly because the Hilbert space of $X\dashind \pi_2$ is
$X\otimes_A \HH_{\pi_2}$.

Moreover, although in \cite{ManRaeSut} the construction of the Doplicher-Roberts algebra $\OO_\rho$ does not explicitly involve an $A$-correspondence,
in the cases where $R=\what H$
the graph $E$ with $E^0=R$ gives a correspondence over $c_0(E^0)$,
and which is Morita equivalent to $A$,
and hence the method outlined in \remref{all correspondences} gives an $A$-correspondence $X$ with $\OO_X$ Morita equivalent to $C^*(E)$, and therefore to $\OO_\rho$.
That being said, at present this observation remains little more than a curiosity.


\begin{thebibliography}{LRRW14}

\bibitem[AM]{AlbMey}
S.~Albandik and R.~Meyer, \emph{{Product systems over Ore monoids}}, Doc. Math. \textbf{20} (2015), 1331--1402.

\bibitem[BKQ]{bkqexelpardo}
E.~B\'edos, S.~Kaliszewski, and J.~Quigg, \emph{{On Exel-Pardo algebras}}, J. Operator Theory, to appear, arXiv:1512.07302 [math.OA].

\bibitem[Dix77]{dixmier}
J.~Dixmier, \emph{{$C^*$-algebras}}, North-Holland, 1977.

\bibitem[EP]{exelpardo}
R.~Exel and E.~Pardo, \emph{{Self-similar graphs, a unified treatment of
  Katsura and Nekrashevych C*-algebras}},
Adv. Math. \textbf{306} (2017), 1046--1129.

\bibitem[Fel64]{fel:weak2}
J.~M.~G. Fell, \emph{{Weak containment and induced representations of groups.
  II}}, Trans. Amer. Math. Soc. \textbf{110} (1964), 424--447.

\bibitem[FMR03]{FowMuhRae}
N.~J. Fowler, P.~S. Muhly, and I.~Raeburn, \emph{Representations of
  {C}untz-{P}imsner algebras}, Indiana Univ. Math. J. \textbf{52} (2003),
  no.~3, 569--605.

\bibitem[FR99]{FowRae}
N.~J. Fowler and I.~Raeburn, \emph{The {T}oeplitz algebra of a {H}ilbert
  bimodule}, Indiana Univ. Math. J. \textbf{48} (1999), no.~1, 155--181.

\bibitem[Gre80]{gre:structure}
P.~Green, \emph{The structure of imprimitivity algebras}, J. Funct. Anal.
  \textbf{36} (1980), no.~1, 88--104.

\bibitem[KPW98]{KajPinWatIdeal}
T.~Kajiwara, C.~Pinzari, and Y.~Watatani, \emph{Ideal structure and simplicity
  of the {$C^\ast$}-algebras generated by {H}ilbert bimodules}, J. Funct. Anal.
  \textbf{159} (1998), no.~2, 295--322.

\bibitem[KPQ12]{graphcorres}
S.~Kaliszewski, N.~Patani, and J.~Quigg, \emph{Characterizing graph
  {$C^*$}-correspondences}, Houston J. Math. \textbf{38} (2012), 751--759.

\bibitem[Kat04]{katsuracorrespondence}
T.~Katsura, \emph{On {$C^*$}-algebras associated with {$C^*$}-correspondences},
  J. Funct. Anal. \textbf{217} (2004), no.~2, 366--401.

\bibitem[Kat08]{KatKirchberg}
\bysame, \emph{A construction of actions on {K}irchberg algebras which induce
  given actions on their {$K$}-groups}, J. Reine Angew. Math. \textbf{617}
  (2008), 27--65.

\bibitem[KP00]{KirPhi}
E.~Kirchberg and N.~C. Phillips, \emph{Embedding of exact {$C^*$}-algebras in
  the {C}untz algebra {$\scr O_2$}}, J. Reine Angew. Math. \textbf{525} (2000),
  17--53.

\bibitem[KPR98]{kpr}
A.~Kumjian, D.~Pask, and I.~Raeburn, \emph{{Cuntz-Krieger algebras of directed
  graphs}}, Pacific J. Math. \textbf{184} (1998), 161--174.

\bibitem[LRRW14]{lrrw}
M.~Laca, I.~Raeburn, J.~Ramagge, and M.~F. Whittaker, \emph{Equilibrium states
  on the {C}untz-{P}imsner algebras of self-similar actions}, J. Funct. Anal.
  \textbf{266} (2014), no.~11, 6619--6661.

\bibitem[LS96]{LSboundary}
M.~Laca and J.~Spielberg, \emph{Purely infinite {$C^*$}-algebras from boundary
  actions of discrete groups}, J. Reine Angew. Math. \textbf{480} (1996),
  125--139.

\bibitem[Mac52]{mackeyinduced}
G.~W. Mackey, \emph{Induced representations of locally compact groups. {I}},
  Ann. of Math. (2) \textbf{55} (1952), 101--139.

\bibitem[MRS92a]{ManRaeSut2}
M.~H. Mann, I.~Raeburn, and C.~E. Sutherland, \emph{Representations of compact
  groups, {C}untz-{K}rieger algebras, and groupoid {$C^*$}-algebras},
  Miniconference on probability and analysis ({S}ydney, 1991), Proc. Centre
  Math. Appl. Austral. Nat. Univ., vol.~29, Austral. Nat. Univ., Canberra,
  1992, pp.~135--144.

\bibitem[MRS92b]{ManRaeSut}
\bysame, \emph{Representations of finite groups and {C}untz-{K}rieger
  algebras}, Bull. Austral. Math. Soc. \textbf{46} (1992), no.~2, 225--243.

\bibitem[MM64]{SurveyMatrix}
M.~Marcus and H.~Minc, \emph{A survey of matrix theory and matrix
  inequalities}, Allyn and Bacon Inc., Boston, 1964, Allyn and Bacon series in
  advanced mathematics.

\bibitem[Mor]{adamtensor}
A.~Morgan, \emph{Cuntz-Pimsner algebras associated to tensor products of
  $C^*$-correspondences},
  J. Aust. Math. Soc. \textbf{102} (2017), no.~3, 348--368.  

\bibitem[MS00]{muhsolmorita}
P.~S. Muhly and B.~Solel, \emph{On the {M}orita equivalence of tensor
  algebras}, Proc. London Math. Soc. (3) \textbf{81} (2000), no.~1, 113--168.

\bibitem[Nek09]{nek}
V.~Nekrashevych, \emph{{C$^*$-algebras and self-similar groups}}, J. Reine
  Angew. Math. \textbf{630} (2009), 59--123.

\bibitem[Phi00]{phillipsclassification}
N.~C. Phillips, \emph{A classification theorem for nuclear purely infinite
  simple {$C^*$}-algebras}, Doc. Math. \textbf{5} (2000), 49--114 (electronic).

\bibitem[Rae05]{raeburngraph}
I.~Raeburn, \emph{Graph algebras}, CBMS Regional Conference Series in
  Mathematics, vol. 103, Published for the Conference Board of the Mathematical
  Sciences, Washington, DC, 2005.

\bibitem[RW98]{tfb}
I.~Raeburn and D.~P. Williams, \emph{{Morita equivalence and continuous-trace
  $C^*$-algebras}}, Math. Surveys and Monographs, vol.~60, American
  Mathematical Society, Providence, RI, 1998.

\bibitem[Rie74]{rie:induced}
M.~A. Rieffel, \emph{{Induced representations of $C^*$-algebras}}, Adv. Math.
  \textbf{13} (1974), 176--257.

\bibitem[Spi12]{spielbergBaumslag}
J.~Spielberg, \emph{{$C^\ast$}-algebras for categories of paths associated to
  the {B}aumslag-{S}olitar groups}, J. Lond. Math. Soc. (2) \textbf{86} (2012),
  no.~3, 728--754.
  
\bibitem[Wil07]{danacrossed}
D.~P. Williams, \emph{Crossed products of $C^*$-algebras},
Math. Surveys and Monographs, vol.~134,
American
Mathematical Society, Providence, RI, 2007.

\bibitem[Zim84]{Zimmer}
R.~J. Zimmer, \emph{Ergodic theory and semisimple groups}, Monographs in
  Mathematics, vol.~81, Birkh\"auser Verlag, Basel, 1984.

\end{thebibliography}

\providecommand{\bysame}{\leavevmode\hbox to3em{\hrulefill}\thinspace}
\providecommand{\MR}{\relax\ifhmode\unskip\space\fi MR }
\providecommand{\MRhref}[2]{%
  \href{http://www.ams.org/mathscinet-getitem?mr=#1}{#2}
}
\providecommand{\href}[2]{#2}

\end{document}